\documentclass[a4paper,final]{amsart}
\usepackage[margin=3cm]{geometry}

\usepackage{amsfonts}
\usepackage{amssymb}
\usepackage[latin2]{inputenc}
\usepackage{enumerate}
\usepackage{amsmath}
\usepackage{amssymb}
\usepackage{framed,xcolor}
\usepackage{lipsum,hyperref}

\colorlet{shadecolor}{yellow!20}

\numberwithin{equation}{section}
\newtheorem{theorem}{Theorem}[section]
\newtheorem{lemma}[theorem]{Lemma}
\newtheorem{claim}[theorem]{Claim}
\newtheorem{conjecture}[theorem]{Conjecture}

\newtheorem{cor}[theorem]{Corollary}

\newcommand{\thistheoremname}{}
\newtheorem*{genericthm*}{\thistheoremname}
\newenvironment{namedthm*}[1]
  {\renewcommand{\thistheoremname}{#1}%
   \begin{genericthm*}}
  {\end{genericthm*}}

\theoremstyle{definition}       {         
\newtheorem{remark}[theorem]{Remark}

}

\newcommand{\Z}{\mathbb{Z}}

\newcommand{\R}{\mathbb{R}}
\newcommand{\C}{\mathcal{C}}
\newcommand{\G}{\mathcal{G}}
\newcommand{\LL}{\mathcal{L}}
\newcommand{\Qo}{\mathcal{Q}}
\newcommand{\Qc}{\overline{\mathcal{Q}}}

\newcommand{\eps}{\varepsilon}

\newcommand{\K}{\mathcal{K}}

\newcommand{\I}{\mathcal{I}}
\newcommand{\J}{\mathcal{J}}

\newcommand{\inte}{\text{int}\,}

\def\beq{\begin{equation}}
\def\eeq{\end{equation}}
\newcommand{\semmi}[1]{{}}

\newcommand{\bp}{\begin{proof}}
\newcommand{\ep}{\end{proof}}

\DeclareMathOperator{\ubdim}{\overline{dim}_M}
\DeclareMathOperator{\hdim}{\dim_H}

\DeclareMathOperator{\diam}{diam}

\DeclareMathOperator{\graph}{graph}

\DeclareMathOperator{\proj}{proj}

\newcommand{\su}{\subset}
\newcommand{\sm}{\setminus}

\renewcommand{\phi}{\varphi}

\newcommand{\mj}[1]{}

\newcommand{\RP}{\mathbb{RP}}

\author{Tam\'as Keleti and Andr\'as M\'ath\'e}

\begin{document}

\title[Equivalent Kakeya conjectures]{Equivalences between different forms of the Kakeya conjecture and
duality of Hausdorff and packing dimensions for additive complements}

\address
{Institute of Mathematics, E\"otv\"os Lor\'and University, 
P\'az\-m\'any P\'e\-ter s\'et\'any 1/c, H-1117 Budapest, Hungary}

\email{tamas.keleti@gmail.com}

\address
{Mathematics Institute, University of Warwick, Coventry, CV4 7AL, UK}

\email{a.mathe@warwick.ac.uk}



\thanks{This research was supported  
by the Hungarian National Research, Development and Innovation Office -- NKFIH, 124749.
The first author is grateful to the Alfr\'ed R\'enyi Institute,
where he was a visiting researcher during a part of this project.
He was also partially supported 
by the Hungarian National Research, Development and Innovation Office -- NKFIH, 104178.
} 

\begin{abstract}
The Kakeya conjecture is generally formulated as one the following statements: every compact/Borel/arbitrary subset of $\R^n$ that contains a (unit) line segment in every direction has Hausdorff dimension $n$; or, sometimes, that every closed/Borel/arbitrary subset of $\R^n$ that contains a full line in every direction has Hausdorff dimension $n$. These statements are generally expected to be equivalent.
Moreover, the condition that the set contains a line (segment) in every direction is often relaxed by requiring a line (segment) for a ``large'' set of directions only, where large could mean a set of positive $(n-1)$-dimensional Lebesgue measure.

Here we prove that all the above forms of the Kakeya conjecture are indeed equivalent. In fact, we prove that there exist $d\le n$ and a compact set $C\subset \R^n$ of Hausdorff dimension $d$ that contains a unit line segment in every direction (and also a closed set of dimension $d$ that contains a line in every direction) such that
every set $S\subset \R^n$ that contains a line segment in every direction of a set of Hausdorff dimension $n-1$, must have dimension at least $d$. 

A crucial lemma we use is the following: For every Borel set $A\subset \R^n$ of 
positive Lebesgue measure
there exists a compact set $E$ of 
upper Minkowski dimension zero 
such that the 
Minkowski sum $A+E$ has non-empty interior.

This lemma can be also used to obtain results on the duality of Hausdorff and packing dimensions via additive complements: For any 
non-empty Borel set $A\su\R^n$ we show that
\[
\begin{split}
\dim_H A 
   & =  n - \inf\{\dim_P B : 
           B \subset \R^n \textrm{ Borel},\ A+B=\R^n\},
           \text{ and}\\
 \dim_P A     
  & =  n - \inf\{\hdim B : 
           B \subset \R^n \textrm{ Borel},\ A+B=\R^n\}.   
\end{split} 
\]
\end{abstract}

\maketitle

\section{Introduction}

\subsection{Equivalent forms of the Kakeya conjecture}
The Kakeya conjecture is a famous open problem in geometric measure theory, related to harmonic analysis and additive combinatorics. It is frequently stated in many different forms, most of which (at least the more conservative ones) are generally expected to be to equivalent to each other. The aim of this paper is to show that many of these are indeed equivalent.
 
Let us agree on what we mean by \emph{the Kakeya conjecture} in this paper. We will say that a subset of $\R^n$ is a \emph{Besicovitch set} if it contains a unit line segment in every direction.
The \emph{Kakeya conjecture} is the statement that every Besicovitch
set has Hausdorff dimension $n$.
This trivially holds for $n=1$, known to be true for $n=2$
(Davies \cite{Da}) and open for $n\ge 3$.

The Kakeya conjecture is sometimes stated for compact Besicovitch sets only. The variant in which
we require full lines instead of line segments also
appears often (in which case, we might want to require the set to be closed).

It is widely believed that all of these variants
are equivalent.
Moreover, it is also believed that the condition of having a line segment in \emph{every} direction can be relaxed: requiring a line segment in a set of directions of positive $(n-1)$-dimensional Lebesgue measure should give rise to an equivalent conjecture. 
Those lower bounds for the Hausdorff dimension of compact 
Besicovitch sets that are obtained via partial 
results for the Kakeya maximal function conjecture
give the very same bounds for compact sets that contain line segments in directions forming a set of positive $(n-1)$-dimensional Lebesgue measure.

In this paper we show that these forms of the Kakeya conjecture are equivalent. In fact, we prove that the Hausdorff dimension of all sets mentioned above (that could feature in a form of the Kakeya conjecture) are minimized by the Hausdorff dimension of a compact Besicovitch set. More precisely, we prove the following theorem.

\begin{theorem}\label{all_new}
For every $n$ there is a unique real number $d\le n$ with the following properties.
\begin{enumerate}
\item There exists a \emph{compact} Besicovitch set of Hausdorff dimension $d$, that is, a compact set that contains a \emph{unit line segment} in \emph{every} direction.
\item There exists a \emph{closed} set of Hausdorff dimension $d$ that contains a \emph{line} in \emph{every} direction.
\item Let $D$ be a set of directions (a subset of $\mathbb{RP}^{n-1}$) of Hausdorff dimension $n-1$, and let $S$ be an arbitrary subset of $\R^n$ that contains a line segment in every direction of $D$. Then $S$ has Hausdorff dimension at least $d$.\end{enumerate}
\end{theorem}

In particular, this theorem implies that for every compact Besicovitch set there is a closed set of the same Hausdorff dimension that contains \emph{lines} in every direction. It appears that this was not known. Our proof of this is rather short, the Baire category theorem in a carefully designed complete metric space gives this result (in Lemma~\ref{newtyp}).

In fact, it turns out that a typical compact Besicovitch set in the Baire category sense (see Section~\ref{s:Kakeya} for the precise setup) has the above minimal Hausdorff dimension $d$. 
This implies that it is enough to solve the Kakeya problem
for compact Besicovitch sets that have the same properties as the 
typical compact Besicovitch sets. 
For example, it is not hard to show that in a typical compact Besicovitch set $B\su\R^n$ the collection $\LL$ of the lines
of the unit line segments of $B$ has Hausdorff dimension 
$n-1$ and packing dimension $2n-2$.
Thus it is enough to prove the Kakeya conjecture
for compact Besicovitch sets with these properties.
Recently, Wang and Zahl \cite{WZ} posed the Sticky 
Kakeya conjecture, which is the $\dim_P\LL=n-1$ special
case of the Kakeya conjecture, and they proved it
for the $n=3$ case.
As we saw above, the $\hdim \LL=n-1$ or the 
$\dim_P \LL=2n-2$
special case would imply the Kakeya conjecture.

Some steps in the proof of Theorem~\ref{all_new} are of course straightforward, and some are not. For example, it is known that given an arbitrary Besicovitch set, there is always a Borel Besicovitch set with the same Hausdorff dimension (see the proof of Lemma~\ref{l:compact}).
Given a Borel Besicovitch set $B$, one could hope that techniques from descriptive set theory could yield a compact Besicovitch set $B'$ of the same dimension, perhaps as a subset of $B$. This is impossible.
Still, this is almost what we do. The Borel Besicovitch set $B$ can assumed to be $G_\delta$ (the intersection of countably many open sets).
Then we can obtain a compact subset $B'\subset B$  that contains unit line segments in directions of $D$, where $D$ is some set of directions of positive $(n-1)$-dimensional measure. So $B'$ is compact, has small dimension, but does not contain line segments in every direction. One question remains: Given a set $B'$ that contains many line segments, can we obtain a set $B''$ (of the same dimension) that contains line segments in all directions? The set of directions, $D$, in general, is a nowhere dense compact set, therefore taking the union of countably many affine images of $B'$ will not help. Still, perhaps surprisingly, the answer is positive. 
The key observation is that the following lemma could be true and it is all we need to complete the proof.

\begin{lemma}\label{topsecret0}
Let $A\subset \R^d$ be measurable with positive Lebesgue measure. Then there is a compact set $E$ of upper Minkowski dimension zero such that $A+E$ has non-empty interior. 
\end{lemma}

The proof of this lemma, perhaps unsurprisingly, is a random construction, building on analogous covering statements for finite sets/groups.

Theorem~\ref{all_new} and the techniques used in this paper have their limits, of course.
The Kakeya maximal function conjecture is generally considered to be stronger than the Kakeya conjecture, so it is no surprise that we do not prove here that the latter implies the former. However, the Kakeya maximal function conjecture implies that certain other ``almost Besicovitch'' sets have full Hausdorff dimension as well, but the techniques in this paper do not seem to handle some of these. For example, this is the case with those compact sets $C\subset \R^n$ for which there is a line in every direction that intersects $C$ in a set of positive $1$-dimensional Lebesgue measure. The Kakeya maximal function conjecture implies that these sets $C$ have Hausdorff dimension $n$. It is unclear if the Kakeya conjecture implies the same. 
(One might think that Lemma~\ref{topsecret0} is applicable here as well and could give us a Besicovitch set of the same dimension as $C$. Unfortunately, this is not the case.)

\semmi{
\begin{theorem}\label{all_new}
Let $d$ be the Hausdorff dimension of a set in $\R^n$ that contains a line segment in a set of directions
of Hausdorff dimension $n-1$.
Then
\begin{enumerate}
\item there is a \emph{compact set} of Hausdorff dimension at most $d$ that contains a \emph{unit line segment} in 
\emph{every} direction;
\item and there is a \emph{closed set}  of Hausdorff dimension at most $d$
that contains a \emph{line} in \emph{every} direction.

\end{enumerate}
\end{theorem}
}

\subsection{Extending line segments to lines and more conjectures.}
We have seen above that the distinction between full lines and line segments does not play any role in the Kakeya conjecture. Indeed, for every set containing line segments in all directions there is another set, of equal Hausdorff dimension, that contains lines in all directions. However, from another perspective, the distinction between lines and line segments is at the heart of the Kakeya conjecture: what happens if we insist on extending the line segments we have to full lines?

In \cite{Ke}, one of us formulated the 
Line segment extension conjecture, which states that for any collection
of line segments the Hausdorff dimension of their union
is the same as the Hausdorff dimension of the union
of the corresponding lines. In \cite{Ke} it is proved
that this holds in the plane and the general case
would imply that any Besicovitch set in $\R^n$
has Hausdorff dimension at least $n-1$ and 
Minkowski dimension $n$.

In this paper, we prove that the following conjecture is \emph{equivalent} to the Kakeya conjecture:

\begin{conjecture}[General Kakeya conjecture]
\label{c:generalkakeya}
Let $E$ be an arbitrary subset of $\R^n$ and let
$D$ be
the set of directions in which $E$ contains a line
segment. 
Then
$$
\hdim E \ge \hdim D +1,
$$
provided $D$ is non-empty.
\end{conjecture}

We claim that Conjecture~\ref{c:generalkakeya} implies the above mentioned Line segment extension conjecture of \cite{Ke}. 
Indeed, let $S$ be the union of a collection of line segments
in $\R^n$, let $L$ be the union of the collection of the corresponding lines, let $s=\hdim L$ and fix $\eps>0$.
Then,
using a result of Mattila \cite[10.10.~Theorem]{Ma} about the Hausdorff dimension of intersections of a given set with planes,
there exists a hyperplane $H$ such that $\hdim L\cap H>s-1-\eps$. 
By taking a projective transformation $\psi$ that maps $H$ to
the hyperplane at infinity (see e.g. \cite{Ke} for a similar argument) we get $\psi(S)$, which contains a 
line segment in a set of directions of Hausdorff dimension $s-1-\eps$. 
Then, using that projective transformations preserve Hausdorff dimension, and Conjecture~\ref{c:generalkakeya}, we obtain
$\hdim S = \hdim\psi(S)\ge s-1-\eps+1=s-\eps=\hdim L-\eps\ge \hdim S-\eps$,
which completes the proof of the claim.

The above mentioned connections can be summarised as
follows.
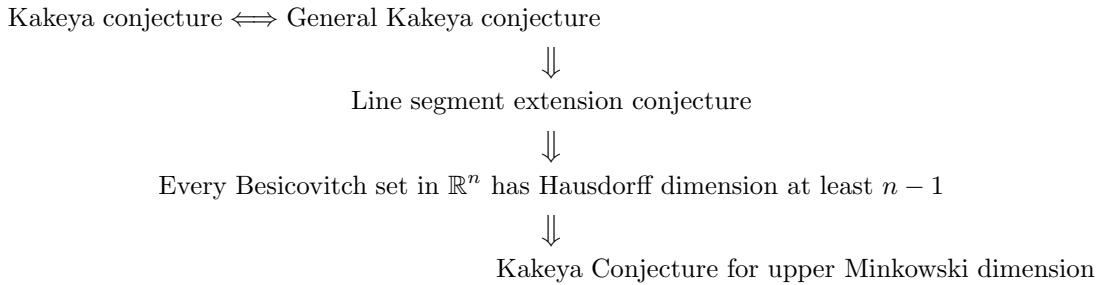
\begin{multline*}
    \text{Kakeya conjecture} \Longleftrightarrow 
    \text{General Kakeya conjecture}\\ 
    \big\Downarrow\\
    \text{Line segment extension conjecture}\\
    \big\Downarrow\\
    \text{Every Besicovitch set in }\R^n \text{ has Hausdorff dimension at least } n-1\\ 
\big\Downarrow\\
\text{Kakeya Conjecture for upper Minkowski dimension}
\end{multline*}

We remark that by modifying the proof in \cite{Wo}
of the fact that the Kakeya maximal function conjecture 
implies the Kakeya conjecture, one can prove directly that
the Kakeya maximal function conjecture also implies Conjecture~\ref{c:generalkakeya}.

The equivalence of Conjecture~\ref{c:generalkakeya} and the Kakeya conjecture actually follows from the following strengthening of Theorem~\ref{all_new}.


\begin{theorem}\label{t:mostgeneral}
For every $n\ge 2$ there exist a compact Besicovitch
set $B\su\R^n$ and a closed set $F\su\R^n$ that 
contains a line in every direction such that
$\hdim B=\hdim F$ and the following holds.

If $E$ is an arbitrary subset of $\R^n$ and
$D$ is a non-empty set of directions such that $E$
contains a line segment in every direction of $D$, then
$$
\hdim E \ge \hdim B - (n-1-\hdim D).
$$

In fact, a typical choice of $B$ and $F$ are good, in the sense of the Baire category theorem, applied in a suitable metric space.

\end{theorem}

\begin{remark}
By a result in \cite{HKM}, if $\LL$ is a non-empty
family of lines that cover a set $E\su\R^n$
such that every line of $\LL$ intersects $E$ in a set of
Hausdorff dimension $1$, then 
$\hdim E = \dim\LL +1$, provided that $\dim\LL\le 1$.
Since the map that assigns to every line its direction is 
Lipschitz, this result implies
Theorem~\ref{t:mostgeneral} for $n=2$ and
also that 
the General Kakeya Conjecture above 
holds whenever $\hdim D\le 1$. 
(Thus a possible way to get closer to 
the General Kakeya Conjecture
and so also to the
Kakeya conjecture is to weaken the restriction 
$\hdim D\le 1$ to $\hdim D\le t$ for larger $t$.
Clearly, the equivalence of the Kakeya conjecture and 
the General Kakeya Conjecture
can be even more useful if one tries to disprove the
Kakeya conjecture.)
\end{remark}

\subsection{Duality between the Hausdorff and 
packing dimension of additive complements}
\label{sub:addcompl}

Our main tool for proving Theorem~\ref{t:mostgeneral} (and thus Theorem~\ref{all_new}) is the following extension of Lemma~\ref{topsecret0}.

\begin{lemma}
\label{topsecret} \ 
\begin{enumerate}
\item 
Let $A\subset \R^n$ be measurable with positive Lebesgue measure. Then there is a compact set $E$ of upper Minkowski dimension zero such that $A+E$ has non-empty interior. 
\item Let $A\subset\R^n$ be a non-empty Borel set of Hausdorff dimension greater than $s$. 
Then 
there is a compact set $E$ of upper Minkowski dimension $n-s$ 
such that $A+E$ has non-empty interior.
\end{enumerate}
\end{lemma}

To prove (1), a universal random set $E$ will be constructed
such that for any fixed set $A$ the random set $E$ has
all the required properties, almost surely.
We will see that (2) follows from (1) by techniques that are now fairly standard.

This second part of Lemma~\ref{topsecret} motivates a spin-off project, unrelated to the Kakeya conjecture. 
The following two theorems show a striking duality between Hausdorff and packing dimensions via additive complements.

\begin{theorem}\label{t:addcompl1}
For every non-empty Borel set $A\subset \R^n$,
\begin{equation}\label{e:dimH}
\begin{split}
\dim_H A & = n - \inf\{\ubdim B \,:\, 
           B \subset \R^n \textrm{ is compact},\  
                          \inte(A+B)\neq\emptyset\}\\
 &  = n - \inf\{\dim_P B \,:\, 
           B \subset \R^n \textrm{ is compact},\  
                          \inte(A+B)\neq\emptyset\}\\     
 & = n - \inf\{\dim_P B \,:\, 
           B \subset \R^n \textrm{ is Borel},\ A+B=\R^n\}.   
\end{split}
\end{equation}
\end{theorem}
\begin{theorem}\label{t:addcompl2}
For any non-empty Borel set $A\subset \R^n$,
\begin{equation}\label{e:dimP}
\begin{split}
\dim_P A  
 &  = n - \inf\{\hdim B : 
           B \subset \R^n \textrm{ compact},\  
                          \inte(A+B)\neq\emptyset\}\\     
 & = n - \inf\{\hdim B : 
           B \subset \R^n \textrm{ Borel},\ A+B=\R^n\}.   
\end{split}
\end{equation}
\end{theorem}

The first of these, Theorem~\ref{t:addcompl1}, follows from Lemma~\ref{topsecret}.
\begin{proof}[Proof of Theorem~\ref{t:addcompl1}]
It is well known (see e.g.~in \cite{Ma}) that for
any Borel sets $U$ and $V$ we have 
$\dim_H(U\times V)\le \dim_H(U) + \dim_P (V)$.
Since $A+B$ is a projection of $A\times B$, 
this implies that whenever $A$ and $B$ are Borel sets
and $A+B=\R^n$ then $\dim_H(A)+\dim_P(B)\ge n$. 
Therefore the last expression is at most $\hdim A$,
so it remains to prove ``$\le$'' instead of the 
equalities. The first one clearly follows from
Lemma~\ref{topsecret} (2), the second from
the well-known fact that 
packing dimension is always less than or equal to the
upper Minkowski dimension,
and the third one from the fact that 
$\R^n$ can be covered by countably many translates
of any set with non-empty interior.
\end{proof}




The proof of Theorem~\ref{t:addcompl2} is nearly identical, but the reference to Lemma~\ref{topsecret} (2) has to be replaced by a reference to the following theorem.

\begin{theorem}\label{t:addpacking}
For any $A\su\R^n$ Borel (or analytic) set with
$\dim_P A>s$ there exists a compact set $B\su \R^n$ 
such that $B$ has zero $(n-s)$-dimensional Hausdorff measure
and $A+B$ has non-empty interior.
\end{theorem}
This theorem will be proved in Section~\ref{s:packing}.

Theorem~\ref{t:addcompl2} is not without history.
The identity \eqref{e:dimP} is closely related to the following identity, which was proved independently
by Bishop and Peres \cite{BP} and by Y.~Xiao \cite{X}.

\begin{theorem}[Bishop--Peres, Xiao (1996)]\label{BPX}
For any non-empty Borel set $A\su\R^n$,
$$
\dim_P(A)=\sup\{\hdim(A\times B)-\hdim B :
B \subset \R^n \textrm{ compact} \}.
$$
\end{theorem}

Note that Theorem~\ref{t:addcompl2} implies
the Bishop--Peres--Xiao theorem. So we actually obtain a new (and in fact much shorter) proof of Theorem~\ref{BPX}.

\begin{remark}
One might wonder whether infimum can be replaced by
minimum in Theorems~\ref{t:addcompl1} and \ref{t:addcompl2}.
We claim that in general the answer is negative for both.

For Theorem~\ref{t:addcompl1} this is witnessed by a set
$A=\cup_{n=1}^\infty A_n$, 
where each $A_n$ is a Borel set with
$\hdim A_n<\hdim A$. Indeed, by the product theorem we used before, for any Borel $B$ with $\dim_P B=n-\hdim A$ 
we have $\dim(A_n+B)<n$, so $A+B=\cup_{n=1}^\infty A_n +B$
is a set of measure zero, so it cannot have non-empty
interior.

By switching $\hdim$ and $\dim_P$ we get a construction for 
Theorem~\ref{t:addcompl2} as well.
\end{remark}

\subsection{Structure of the paper}
The paper is organised as follows.
In Section~\ref{s:Kakeya} we prove Theorem~\ref{t:mostgeneral} using Lemma~\ref{topsecret}.
We prove Theorem~\ref{t:addpacking} in 
Section~\ref{s:packing}, and Lemma~\ref{topsecret} in
Section~\ref{s:keylemma}. 

\subsection{Fractal percolation and non-empty interior of sumsets.} After the initial publication of this work, Pablo Shmerkin informed us that part (2) of Lemma~\ref{topsecret} follows from his results with Ville Suomala on spatially independent martingales and fractal percolation \cite{SV}. Indeed, part (2) of Lemma~\ref{topsecret} is a very special case of \cite[Theorem 13.1]{SV} (modulo an unimportant technical condition). On the other hand, the first half of Lemma~\ref{topsecret} (that is, Lemma~\ref{topsecret0}) does not seem to follow from their results.

In Section~\ref{s:keylemma} we reduce part (2) of Lemma~\ref{topsecret} to part (1) in five sentences, and then prove part (1).
Our direct proof of Lemma~\ref{topsecret} is much shorter than the proof of the very general result of Shmerkin--Suomala \cite{SV}.

\section{Proof of Theorem~\ref{t:mostgeneral}
using Lemma~\ref{topsecret}}
\label{s:Kakeya}

\subsection{Baire category argument in the compact case}

Our first goal is proving the following.

\begin{lemma}\label{newtyp}
For any non-empty compact set of
directions $E\su\RP^{n-1}$ the following hold:

(1) Among those compact subsets of $\R^n$ that contain 
a unit line segment in every direction in $E$
there exists one, $B_{\min}$,
with minimal Hausdorff dimension.

(2) There exists a closed set $F\su\R^n$ that
contains a line in every direction in $E$ 
and $\hdim F=\hdim B_{\min}$.
\end{lemma}
 
To prove Lemma~\ref{newtyp} we use 
a Baire category argument. 
We use a setup similar to that in \cite{CCHK}.
Fix $n\ge 1$. For non-empty compact sets $C\su\R^n$ and 
$E\su\RP^{n-1}$ let
$$
\K(C,E)=\{K\su C\times E \textrm{ compact}: 
\proj_2 K = E \},
$$
where $\proj_2$ denotes projection to the second 
coordinate.
Note that $\proj_2 K = E$ means that
for each $e\in E$ there is a $p\in C$ with $(p,e)\in K$. 

By fixing a natural metric on $\RP^{n-1}$
and then on the product $\R^n\times \RP^{n-1}$ we equip 
$\K(C,E)$ with the Hausdorff metric. Clearly, $\K(C,E)$
is a closed subset of the space of all compact subsets of
$C \times E$, and hence it is a complete metric space.
In particular, Baire category theorem holds for $\K(C,E)$,
so we can speak about a typical $K\in \K(C,E)$ in 
the Baire category sense: a property $P$ holds for a typical
$K\in\K(C,E)$ if 
$$\{K\in\K(C,E)\ :\ P \textrm{ holds for } K\}$$ is residual
in $\K(C,E)$, or equivalently, if there is a dense 
$G_\delta$ set $\G\su\K(C,E)$ such that the property 
holds for every $K\in\G$.

For $K\in\K(C,E)$ and $\ell\in(0,\infty]$ let 
$$
S_\ell(K)=\bigcup_{(p,e)\in K} s(p,e,\ell),
$$
where $s(p,e,\ell)$ denotes the closed line segment
of length $\ell$ with midpoint $p$ and direction $e$
for finite $\ell$, and the line through $p$ and direction 
$e$ for $\ell=\infty$. Clearly, $S_\ell(K)$ is a compact set that contains a line segment  in any direction $e\in E$
for finite $\ell$, and it a closed set which contains
a line in every direction for $\ell=\infty$. 
In particular, 
 $S_1(K)$ is a compact Besicovitch set for any $K\in\K(C,\RP^{n-1})$ for any compact $C\su\R^n$.

By the above observations, the last claim of the
following lemma completes the proof of 
Lemma~\ref{newtyp} by taking $B_{\min}=S_1(K)$
and $F=S_\infty(K)$ for a typical
$K\in\K([0,1]^n, E)$.

\begin{lemma}\label{l:typical}
Let $E\su\RP^{n-1}$ be a non-empty compact set of
directions and let 
$B\su\R^n$ be a compact set which contains a unit line
segment in every direction $e\in E$. 
Then the following statements hold.
\begin{enumerate}
\item 
For any $\ell\in(0,\infty)$ and compact set $C\su\R^n$
with non-empty interior there exists $K\in\K(C,E)$ such 
that $\dim_H S_\ell(K)\le \dim_H B$.
\item
For a typical $K\in\K([0,1]^n, E)$ we have
$\dim_H S_{\ell}(K)\le \dim_H(B)$ for every $\ell\in(0,\infty]$.
\item
For a typical $K\in\K([0,1]^n, E)$ we have
$\dim_H S_{\ell}(K)\le \dim_H(B')$ for every $\ell\in(0,\infty]$
and any compact set $B'\su\R^n$ which contains a unit line
segment in every direction $e\in E$.
\end{enumerate}
\end{lemma}

\begin{proof}
%
(1) 
Let $B(a,r)$ be a ball contained in $C$. 
Choose finitely many compact subsets 
$B_1,\ldots,B_k$ of $B$ such that every unit line segment
in $B$ is contained in at least one of the sets $B_i$, and
for each $i$ the midpoints of the unit line segments are contained
in a ball of radius $r/\ell$. By translating each $B_i$
such that the center of the corresponding ball goes to $a$,
taking their union and magnifying by $\ell$ from $a$
we get a compact set $B'$ of Hausdorff dimension
at most $\dim_H B$ 
which contains a line segment
of length $\ell$ with midpoint in $B(a,r)\su C$ in every direction.
Then letting $K=\{(p,e)\in C\times E\ : s(p,e,\ell)\su B'\}$
we obtain $K\in\K(C,E)$ such 
that $\dim_H S_\ell(K)\le \dim_H B'\le \dim_H B$.
(A similar argument can be found in \cite{FOR}.)

%
%

(2)
Since $S_{\infty}=\cup_{\ell=1}^\infty S_\ell$ and 
countable intersection of residual sets is residual,
it is enough to prove the claim for finite $\ell$.

Let $d=\dim_H (B)$ and let
$$
\G=\{K \in \K([0,1]^n, E) \ :\ \dim_H S_\ell(K)\le d\}.
$$
A standard straightforward argument 
(see e.g.~in \cite[Proof of Lemma 2.2]{CCHK}) gives
that $\G$ is a $G_\delta$ set in $\K([0,1]^n, E)$. 
Thus it is enough to prove that $\G$ is dense in 
$\K([0,1]^n, E)$.

So fix $\delta>0$ and $K\in \K([0,1]^n, E)$. 
We need to find a $K'\in \G$ with $d_H(K',K) < \delta$. 
Choose finitely many compact sets $C_i\su [0,1]^n$ and
$E_i\su E$ all with diameter less than $\delta/2$ such that
$$
K\subset \bigcup_{i=1}^k C_i \times E_i
$$
and that $(C_i\times E_i)\cap K$ and $\inte C_i$ are
non-empty for each $i$.

For each i, by applying the first part of the lemma for
$C=C_i$ and $E=(\proj_2 K)\cap E_i$,
we obtain a non-empty compact set 
$K_i\su C_i\times E_i$ such that 
 $\proj_2 K_i=(\proj_2 K) \cap E_i$ 
and $\dim_H S_\ell(K_i)\le d$.
Then $K'=K_1\cup\ldots \cup K_k$ is a set in $\G$ with 
$d_H(K',K) < \delta$. 
%
%
%

(3) Let $s$ be the infimum of the Hausdorff
dimension of those compact sets that 
contain a unit line segment in every
direction of $E$. Let $B_k$ be a sequence
of such sets with $\hdim B_k\to s$.
Then applying (2) to each $B_k$ and using
that having countably many typical property
is still typical we obtain (3).
\end{proof}

Applying (3) of Lemma~\ref{l:typical} for $E=\RP^{n-1}$
we obtain 
the following result.

\begin{cor}\label{c:minimal}
Among compact Besicovitch sets, there exists one with minimal Hausdorff dimension. \end{cor}

This result is not new, although perhaps it has not been published before. Its proof does not require the specific complete metric space we used here. Instead, one could just use the space of compact Besicovitch sets in the Hausdorff metric, and prove that a typical set has minimal Hausdorff dimension. (This follows the general rule that typical sets, in any setting, always have minimal Hausdorff dimension.) A similar space was considered by T.~K\"orner \cite{Korner}; he proved that a typical compact Besicovitch set has zero Lebesgue measure, giving a new proof for the existence of Besicovitch sets of zero measure.

\subsection{Getting compactness}
The next step is to make a compact construction from
a general one.

\begin{lemma}\label{l:compact}
Let $B\su\R^n$ 
be an arbitrary set which
contains a  line segment in a set of directions 
of positive outer measure. 
Then there exists a compact set $E\su \R^n$
with $\hdim E \le \hdim B$ 
which contains a 
line segment 
in a compact set of directions of positive
measure.

We can also guarantee the existence of a line
such that the projection of each of these line
segments to this line has at least a fixed length.

\end{lemma}

\begin{proof}
We can choose two parallel hyperplanes such that the direction of those line segments in $B$ that have endpoints on these hyperplanes span a set of positive $(n-1)$-dimensional Lebesgue outer measure. 
We may assume that these
parallel hyperplanes are $\{0\}\times\R^{n-1}$ and
$\{1\}\times\R^{n-1}$.
By the regularity of the Hausdorff measure, $B$ is contained
in a $G_\delta$-set $B'$ of the same Hausdorff dimension.
Let
$$
X=\{(a,b)\times \R^{n-1} \times \R^{n-1} : S(a,b)\su B'\},
$$
where
$$
S(a,b)=\{(t,at+b) : t\in [0,1]\}.
$$
It is straightforward to check that the fact that $B'$
is $G_\delta$ implies that $X$ is $G_\delta$ as well.
Let $A'=\proj_1 X$. Being a projection of a Borel set, $A'$ is analytic, therefore Lebesgue measurable.
By our assumptions about $B$ and $B'$, and since $a$ determines
the direction of $S(a,b)$, the set $A'$ has positive
($n-1$-dimensional) measure. 

Applying the Jankov--von Neumann theorem 
(see e.g.~in \cite{Kec}) to $X\su\R^{n-1}\times\R^{n-1}$
we get a Lebesgue measurable function 
$f:A'\to\R^{n-1}$ with $\graph(f)\su X$.
By Luzin's theorem there exists a compact set $A\su A'$
of positive measure such that the restriction 
$g=f|A: A\to\R^{n-1}$ is continuous.
Let 
$$
E=\bigcup_{a\in A} S(a,g(a)).
$$
Since $g$ is continuous and $A$ is compact, 
$E$ is compact as well.
Note also that by construction, $E\su B'$,
so $\hdim E\le \hdim B'=\hdim B$.
Finally, clearly the projection of each 
$S(a,g(a))$ to the first axis has unit length.
\end{proof}

\subsection{Getting all directions}

For $a,b\in\R^{n-1}$ consider the line and its ``slope''
\begin{equation}\label{e:l(a,b)}
\ell(a,b)=\{(t,at+b) : t\in \R\} 
\qquad \textrm{and} \qquad
D(\ell(a,b))=a.
\end{equation}
Note that this way we get exactly those lines of $\R^n$ that 
are not perpendicular to the first coordinate axis.
Clearly $D(\ell(a,b))$ determines the direction of the line
and in fact there is a locally Lipschitz bijection between
the directions and the $D$ of these lines.

The following lemma explains the setting in which Lemma~\ref{topsecret} will be applied. Basically,  given a set containing many line segments, we will use Lemma~\ref{topsecret} to construct a not-much-larger set that contains even more line segments. 

\begin{lemma}\label{l:directionaddition}
Let $\LL$ be a collection of lines in $\R^n$ such that 
none of them are perpendicular to the first coordinate axis.
Let $E\su\R^n$ be arbitrary and $C\su\R^{n-1}$ be 
a bounded set.

Then there exist a collection $\LL'$ of lines in $\R^n$
and a set $E'\su\R^n$ with the following properties.

(1) $\hdim E' \le \hdim E + \ubdim C$.

(2) $D(\LL')=D(\LL)+C$.

(3) For every $\ell'\in\LL'$ there exists $\ell\in\LL$
such that $\proj_1(E\cap \ell)\su \proj_1(E'\cap\ell')$.

Furthermore, if $E$ and $C$ are compact, then so is $E'$.
\end{lemma}

\begin{proof}
Let 
$$
h((t,u),v)=(t,u+tv) \qquad (t\in[0,1], u,v\in\R^{n-1} )
$$
and
$$
E'=h(E\times C).
$$
If $E$ and $C$ are compact then $E'$ is clearly compact.

Using that $h:((\R\times\R^{n-1})\times\R^{n-1})\to\R^n$
is locally Lipschitz and
a well known dimension estimate on product sets
(see e.g.~in \cite{Ma}), we get
$$
\dim_H E'\le \dim_H (E\times C) \le 
\dim_H E + \ubdim C,
$$
which gives (1).

By defining
$$
\LL'=\{\ell(a+c,b) : \ell(a,b)\in\LL, \ c\in C\},
$$
(2) clearly holds.

To prove (3) let $\ell'\in\LL'$. 
Then $\ell'=\ell(a+c,b)$ for some $\ell(a,b)\in\LL$
and $c\in C$.
It is enough to show that
$$
\proj_1(E\cap \ell(a,b))\su \proj_1(E'\cap\ell(a+c,b)).
$$
So let $t\in \proj_1(E\cap \ell(a,b))$.
Then $(t,at+b)\in E$. 
Thus $h((t,at+b),c)\in E'$.
Since 
$$
h((t,at+b),c)=(t,at+b+ct)\in\ell(a+c,b)
$$ 
as well,
this implies that indeed $t\in \proj_1(E'\cap\ell(a+c,b)$,
which completes the proof.
\end{proof}

Now we can prove,  from Lemma~\ref{topsecret}, the following weak version of
Theorem~\ref{all_new}.

\begin{theorem}\label{all}
Let $d$ be the Hausdorff dimension of a set in $\R^n$ that contains a line segment in a set of directions
with positive outer measure.
Then
\begin{enumerate}
\item there is a \emph{compact set} of Hausdorff dimension at most $d$ that contains a \emph{unit line segment} in every direction;
\item and there is a \emph{closed set}  of Hausdorff dimension at most $d$
that contains a \emph{line} in every direction.
\end{enumerate}
\end{theorem}


\begin{proof}

By Lemma~\ref{newtyp}, it is enough to prove 
the first part of the theorem.

Let $B\subset \R^n$ be a set of Hausdorff dimension $d$ that contains a line segment in a set of directions of 
positive outer measure.
By Lemma~\ref{l:compact}, there is a compact set $E$ with $\hdim E\le \hdim B$ that contains a line segment
in 
a compact set of directions of 
positive measure 
such that the projections of these line segments
to a fixed line has at least a fixed length.
We may assume that this line is the first axis
and this length is $1$.

Let $\LL$ be the collection of the lines of these
line segments. 
Then, using the notation \eqref{e:l(a,b)}, 
$D(\LL)$ is a compact set of positive 
($n-1$-dimensional Lebesgue) measure
and for every $\ell\in\LL$ the projection
$\proj_1(E\cap\ell)$ contains an interval of length $1$.

By Lemma~\ref{topsecret},
there exists a compact $C\su\R^{n-1}$ with
$\ubdim C=0$ and $\inte(D(\LL)+C)\neq\emptyset$.
Let $\LL'$ be the collection of the lines and $E'$ be the set
we obtain by applying Lemma~\ref{l:directionaddition}.

Then $E'$ is compact, 
$$
\hdim E'\le \hdim E + \ubdim C \le \hdim B + 0 =d,
$$
and every line of $\LL'$ intersects $E'$ in a line segment
of length at least $1$.
Since $D(\LL')=D(\LL)+C$ has non-empty interior
this implies that 
a suitable finite union of rotated copies 
of $E'$ is a compact set of 
Hausdorff dimension at most $d$
which contains a unit line segment in every direction.
\end{proof}

\subsection{Besicovitch set from an arbitrary set and 
the proof of Theorem~\ref{t:mostgeneral}}

First we construct a Besicovitch set from an
arbitrary set. 
Note that the following result clearly implies that
the General Kakeya Conjecture (Conjecture~\ref{c:generalkakeya})
is equivalent to the Kakeya conjecture.

\begin{theorem}\label{t:segments}
Let $E$ be a subset of $\R^n$ and let $D$ be
the set of directions in which $E$ contains a line
segment. If $D\neq\emptyset$ then there exists a 
compact Besicovitch set $B$ in $\R^n$ with
$$
\hdim B \le n-1 + \hdim E - \hdim D.
$$
\end{theorem}

%

\begin{proof}
By suitably rotating $E$ we can suppose that
removing from $D$ the directions that are perpendicular to
the first coordinate axis does not decrease the
Hausdorff dimension of $D$.
Let $\LL$ be the collection of those lines in $\R^n$
that are not perpendicular to the first coordinate axis
and intersect $E$ in a set that contains a line segment.
By our assumptions and using  notation \eqref{e:l(a,b)}, 
$\hdim D(\LL)=\hdim D$. 

Fix $s<\hdim D(\LL)=\hdim D$. 
By applying (2) of Lemma~\ref{topsecret} for 
$D(\LL)\su\R^{n-1}$ we obtain
a compact set $C$ with $\ubdim C = n-1-s$ such that
$D(\LL) + C$ has non-empty interior. 
Let $\LL'$ be the collection of the lines and $E'$ be the set
we obtain by applying Lemma~\ref{l:directionaddition}.
Then 
\begin{equation}
\hdim E'\le \hdim E + n-1-s,
\end{equation}
and $E'$
contains line segments with directions in a set of positive measure. 
Then, by Theorem~\ref{all}, this implies that 
there is a compact
Besicovitch set of Hausdorff
dimension at most $\hdim E + n-1-s$.
Since $s<\hdim D$ was arbitrary and 
by Corollary~\ref{c:minimal} there exists a 
compact Besicovitch set in $\R^n$
with minimal Hausdorff dimension,
this implies that in $\R^n$
there exists a compact Besicovitch set
of Hausdorff dimension at most $n-1+\hdim E-\hdim D$.
\end{proof}

Now it is easy to prove our most general result,
which clearly implies Theorem~\ref{all_new} as well.

\begin{proof}[Proof of Theorem~\ref{t:mostgeneral}]
Let $B$ and $F$ be the (typical) sets we obtain when 
Lemma~\ref{newtyp} is applied to $\RP^{n-1}$
(as $E$).
Then clearly $B$ is a compact Besicovitch set with 
minimal Hausodorff dimension, $F$ is a closed
set which contains a line in every direction
and $\hdim F=\hdim B$. 

Let $E$ and $D$ be as in the statement of 
Theorem~\ref{t:mostgeneral}.
By Theorem~\ref{t:segments} there exists
a compact Besicovitch set $B_1\su\R^n$ such that
$\hdim B_1\le n-1+\hdim E-\hdim D$.
Since $B$ is a compact Besicovitch set with
minimal Hausdorf dimension, $\hdim B\le \hdim B_1$,
which completes the proof.
\end{proof}

\section{Proof of Theorem~\ref{t:addpacking}}
\label{s:packing}

The proof is based on a variant of a theorem
of Lorentz \cite{Lo} in additive combinatorics.
For completeness we present a full proof of the
statement we need, by using a simple argument of I.~Z.~Ruzsa \cite{Ru}. 
The following lemma and its short proof
is a straightforward modification of Lemma~5.2 of
\cite{Ru}, which is a more general statement
but only in $\Z_q$.

\begin{lemma}[Ruzsa]\label{l:Ruzsa}
Let $G$ be a finite Abelian group of order $q$
and let $A$ be a subset of $G$ with cardinality
$|A|\ge tq$ for some $0<t<1$ and let
$k=\left\lceil (\log q)/{t} \right\rceil$.

Then there exists $B\su G$ with
cardinality $k$ such that $A+B=G$.
\end{lemma} 

\begin{proof}
The proof is the same as in \cite{Ru}, it is
included for completeness. 
Choose $B$ as a random $k$-element subset
of $G$, by taking each $q \choose k$ $k$-element
subset of $G$ with equal probability.
Fix $g\in G$.
Since $g\not\in A+B$ if and only if $B\cap (g-A)=\emptyset$, the probability of this event is 
\begin{equation}\label{e:binom}
{{q-|A|} \choose k }/ {q \choose k} 
\le {{q(1-t)} \choose k }/ {q \choose k}
 < (1-t)^k.
\end{equation}
Thus the expected value of $|G\sm (A+B)|$ is
less than $(1-t)^k q \le e^{-tk} q \le 1$, so
with positive probability $A+B=G$. 
\end{proof}

\begin{lemma}[Higher dimensional variant
of a theorem of Lorentz \cite{Lo}]
\label{l:Lorentz}
Let $U\subset\{0,1,\ldots,m-1\}^n$ with
$|U|\ge m^\alpha$. 
Then there exists 
$V\subset \{-(m-1), -(m-2), \ldots, m-1\}^n$ 
such that 
$U+V\supset \{0, 1,\ldots, m-1\}^n$ 
and
$|V|\le 2^n \lceil n \cdot m^{n-\alpha} \cdot \log m\rceil$.
\end{lemma}

\begin{proof}
We apply Lemma~\ref{l:Ruzsa} to 
$G=\Z_m^n, q=m^n,t=m^{\alpha-n}$ to obtain a set
$B\su \Z_m^n$ such that $U+B=\Z_m$ and 
$|B|\le \left\lceil(\log m^n)/m^{\alpha-n}\right\rceil
=\lceil n \cdot m^{n-\alpha} \cdot \log m\rceil$.
By taking $2^n$ copies of $B$ we get a 
$V\su \{-m, -(m-1), \ldots, m-1\}^n$ with the right
cardinality and with $U+V\supset \{0, 1,\ldots, m-1\}^n$.
Note that for the last containment any $v\in V$ with
any $-m$ coordinate would
be superfluous, so we can can suppose that 
$V\subset \{-(m-1), -(m-2), \ldots, m-1\}^n$, which
completes the proof.
\end{proof}

The following technical lemma would be obvious
without taking the interiors of the cubes, this
version requires some care.

\begin{lemma}
\label{l:correcterror}
If $Q\su\R^n$ is a closed cube, $A\su \inte Q$ and
$\ubdim A>s$ then there exists arbitrarily large $m$
such that subdividing $Q$ into $m^n$  closed
cubes of equal size, the interior of at least $m^s$ of them
intersect $A$. 
\end{lemma}

\begin{proof}
\semmi{
Let $\Qc_j$ denote the set of closed cubes we get
by subdividing $Q$ into $m^j$ 
equal cubes, 
let $\Qo_j$ consist of the interiors of the cubes
of $\Qc_j$ and let $G_j$ be the union of the 
open cubes of $\Qo_j$.
The condition $\ubdim A>s$ implies that there exists
arbitrarily large $N$ such that 
$k\ge 2^{(n+1)s}\cdot (n+1)\cdot 2^n \cdot N^s$ of 
the closed cubes of $\Qc_N$ intersect $A$.
Let $a_1,\ldots,a_k\in A$ be chosen from 
different cubes of $\Qc_N$.
By Chebyshev's theorem there exist distinct
primes $p_1,\ldots,p_{n+1}$ between $N$ and 
$M=2^{n+1}N$.
Then $A\subset\inte Q=\cup_{i=1}^{n+1} G_{p_i}$, 
so there exists
an $m=p_i$ such that $G_m$ contains at least 
$\frac{k}{n+1}$ points of $\{x_1,\ldots,x_k\}\su A$.
Since any open cube of $\Qo_m$ intersects at
most $2^n$ cubes of $\Qc_N$ we get that at least
$\frac{k}{(n+1)2^n}\ge M^s\ge m^s$ open cubes of 
$\Qo_m$ intersect $A$.
}
Let $\Qc_m$ denote the set of closed cubes we get
by subdividing $Q$ into $m^n$ grid cubes of equal size, 
let $\Qo_m$ consist of the interiors of the cubes
of $\Qc_m$ and let $G_m$ be the union of the 
open cubes of $\Qo_m$.
The condition $\ubdim A>s$ implies that for every constant $C$ there are arbitrarily large $N$ such that $A$ contains at least $CN^s$ points that are covered by different cubes of $\Qc_N$. Let these be $\{x_1,\ldots, x_k\}$, where $k\ge CN^s$. We will use this for $C=(n+1)2^{n} 2^{(n+1)s}$.

By Chebyshev's theorem there exist distinct
primes $p_1,\ldots,p_{n+1}$ between $N$ and 
$M=2^{n+1}N$.
Then $A\subset\inte Q=\cup_{i=1}^{n+1} G_{p_i}$, 
so there exists
an $m=p_i$ such that $G_m$ contains at least 
$\frac{CN^s}{n+1}$ points of $\{x_1,\ldots,x_k\}\su A$.
Since any open cube of $\Qo_m$ intersects at
most $2^n$ cubes of $\Qc_N$, we get that at least
$\frac{CN^s}{(n+1)2^n}\ge M^s\ge m^s$ open cubes of 
$\Qo_m$ intersect $A$.
\end{proof}

\begin{lemma}
\label{l:zeroHausdorff}
Let $A\su(0,1)^n$ be a non-empty 
compact set such that $\ubdim(A\cap Q)>s$ for any
open cube $Q$ that intersects $A$.
Then there exists a compact set $B\su[0,2]^n$ such
that $A+B\supset [1,2]^n$ and the $n-s$-dimensional
Hausdorff measure of $B$ is zero.
\end{lemma}

\begin{proof}
For every $k=0,1,\ldots$ we construct a
finite collection $\I_k$
of cubes in $[0,1]^n$, a 
finite collection $\J_k$ of cubes
in $[0,2]^n$ and a set $E_k \su \I_k \times \J_k$
such that for every $k$ we have the following:
\begin{enumerate}
\item\label{p:length}
The sidelength of each cube in $\I_k$ and $\J_k$ is
at most $2/2^k$.
\item\label{p:A}
The interior of any $I\in \I_k$ intersects $A$.
\item \label{p:tree}
Every cube $J\in \J_k$ is a subset of some cube 
$J'\in \J_{k-1}$ (if $k\ge 1$).
\item \label{p:sum}
$$
\sum_{J\in\J_k} (\diam J)^{n-s} < (2\sqrt{n})^{n-s-k} .
$$
\item \label{p:double}
If 
$(I,J)\in E_k$
then the sidelength of $J$ is twice
the sidelength of $I$.
\item \label{p:cover}
$$
[1,2]^n \su \bigcup_{(I,J)\in E_k}\frac{1}{3}(I+J),
$$
where 
$\frac{1}{3}(I+J)$ denotes the cube we obtain by shrinking
the cube $I+J$ from its center by $1/3$.
\end{enumerate}

First we show that if we have a construction with
the above properties then 
$$
B=\bigcap_{k=0}^\infty \bigcup_{J\in\J_k} J
$$ 
satisfies all the requirements.
It is clear that $B$ is a compact subset of $[0,2]^n$.
By \eqref{p:sum} we get that the 
$n-s$-dimensional Hausdorff measure of $B$ is zero.
To prove that $[1,2]^n\su A+B$, let $x\in [1,2]^n$.
By \eqref{p:cover}, for each $k$ there exists 
$u_k\in \cup_{I\in\I_k} I$ and 
$v_k\in \cup_{J\in\J_k} J$ such that $u_k+v_k=x$.
By taking convergent subsequences
and using \eqref{p:length} and \eqref{p:A}, 
we get limits $u\in A$
and $v\in B$ such that $u+v=x$.

We construct $\I_k, \J_k$ and $E_k$ with properties (1)-(6)
by induction. Let $\I_0=\{[0,1]^n\}, \J_0=\{[0,2]^n\}$
and $E_0=\I_0\times \J_0$. Suppose that $k\ge 1$ and 
we have already defined $\I_k, \J_k$ and $E_k$ such that
(1),..,(6) hold. 
Let $\eps=(2\sqrt{n})^{n-s-(k+1)}/|E_k|$.
For each $e=(I,J)\in E_k$ we will choose a collection $\I^e$ of subcubes of $I$ and a collection $\J^e$ of subcubes of
$J$ and we will define $\I_{k+1}=\cup_{e\in E_k}\I^e$,
$\J_{k+1}=\cup_{e\in E_k}\J^e$ and 
$E_{k+1}=\cup_{e\in E_k}\I^e\times\J^e$.
Then \eqref{p:tree} clearly holds. 
To show \eqref{p:sum} for $k+1$ it is enough to show that for each
$e\in E_k$ we have 
\begin{equation}\label{e:sum}
\sum_{J'\in \J^e} (\diam J')^{n-s}<\eps.
\end{equation}
To show that \eqref{p:cover} for $k$ implies that \eqref{p:cover} also holds for $k+1$ it is enough to prove that for
any $e=(I,J)\in E_k$ we have
\begin{equation}\label{e:cover}
\frac{1}{3}(I+J)\su \bigcup_{(I',J')\in \I^e\times \J^e}\frac{1}{3}(I'+J').
\end{equation}
 
Fix $=(I,J)\in E_k$. By \eqref{p:double} there exists
$a,b\in\R^n$ and $h>0$ such that $I=a+[0,h]^n$ and 
$J=b+[-h,h]^n$. 
Note that $\frac{1}{3}(I+J)=\frac{1}{3}(a+b+[-h,2h]^n)=
a+b+[0,h]^n$.

By \eqref{p:A} and the assumption of the lemma we have
$\ubdim(A\cap \inte I)>s$, 
so we can fix $s'$ such that $s<s'<\ubdim(A\cap \inte I)$.
Since $s<s'$ there exists an $m_0\ge 2$ such that
for any $m\ge m_0$ we have
\begin{equation}\label{e:epsilon}
2^n \lceil n\cdot m^{n-s'} \cdot \log m\rceil
\left(2\sqrt{n}\frac{h}{m}\right)^{n-s}<\eps.
\end{equation}
By Lemma~\ref{l:correcterror},
there exists $m\ge m_0$
such that subdividing $I$ into $m^n$ equal closed subcubes,
the interior of at least $m^{s'}$ of them intersect $A$.
Let $\I^e$ consist of these closed subcubes. 
Then \eqref{p:A} will automatically hold for $k+1$. 

Let
$$
U=\{u\in\{0,\ldots,m-1\}^n\ :\ \big(a+(h/m)u+(0,h/m)^n 
\big)\cap A\neq\emptyset\}.
$$
Then 
$$
\I^e=\{a+(h/m)u+[0,h/m]^n\ :\ u\in U\},
$$
so
$|U|=|\I^e|\ge m^{s'}$.
By applying Lemma~\ref{l:Lorentz} to $U$ we get 
$V\su\{-m+1,\ldots,m-1\}^n$ such that 
$U+V\supset \{0, 1,\ldots, m-1\}^n$ 
and
\begin{equation}\label{e:Vsize}
|V|\le 2^n \lceil n \cdot m^{n-s'} \cdot \log m\rceil.
\end{equation}
Then let 
$$
\J^e=\{b+(h/m)v+[-h/m,h/m]^n\ :\ v\in V\}.
$$
Then \eqref{p:length} and \eqref{p:double} clearly
holds for $k+1$, so it remains to show \eqref{e:sum}
and \eqref{e:cover}.
Note that \eqref{e:sum} follows from \eqref{e:epsilon}
and \eqref{e:Vsize}.

It remains to check \eqref{e:cover}. 
Note that 
$$
\{(1/3)(I'+J')\ :\ (I',J')\in\I^e\times J^e\}=
\{a+b+(h/m)(u+v)+[0,h/m]^n\ :\ u\in U, v\in V\}.
$$
Since $U+V\supset \{0, 1,\ldots, m-1\}^n$, the union
of these cubes indeed cover 
$\frac{1}{3}(I+J)=a+b+[0,h]^n$,
which completes the proof of \eqref{e:cover}
and the lemma.
\end{proof}


\begin{proof}[Proof of Theorem~\ref{t:addpacking}]
By a result of H.~Joyce and D.~Preiss \cite{JP}
there exists a compact set $A'\su A$ with $\dim_P A'>s$.
We can clearly suppose that $A'\su (0,1)^n$. 
Let $\C$ be the collection of those open 
cubes that intersect $A'$ in a set of packing dimension
at most $s$, and let $\C'$ be those cubes in $\C$ that
have vertices with rational coordinates only. 
Then cubes in $\C$ and $\C'$ have the very same
union $U$ and, since $\C'$ is countable, 
$\dim_P(A'\cap U)\le s$. Thus $A''=A'\sm U$ is a 
non-empty compact
subset of $[0,1]^n$ such that for any open cube $Q$
that intersects $A''$,
$\ubdim(Q\cap A'')\ge \dim_P (Q\cap A'')>s$.
Therefore Lemma~\ref{l:zeroHausdorff} can be applied
to $A''$.
\end{proof}

\section{Proof of Lemma~\ref{topsecret}}
\label{s:keylemma}

Our first goal is to show that (1) implies (2) in Lemma~\ref{topsecret}. 
For this we need the following lemma.
This argument was also used
in \cite[Theorem 2.9]{Ke2}. 

\begin{lemma}\label{standard}
For every $s,t\in[0,n]$ and non-empty
Borel set $A\subset \R^n$ of Hausdorff dimension $s$ there is a compact
set $B\subset \R^n$ of upper Minkowski dimension $t$ such that $A+B$ has Hausdorff dimension $\min(s+t, n)$. If $s+t>n$, it has positive Lebesgue measure.
\end{lemma}
\begin{proof}
Let $B_1$ be a compact set (for example a self-similar set)
with $\hdim B_1=\ubdim B_1=t$.
Then $\hdim (A\times B_1)=s+t$, so by the projection 
theorem of Marstrand and Mattila, its orthogonal
projection to almost every $n$-dimensional subspace
of $\R^{2n}$ has Hausdorff $\min(s+t,n)$ and it has
positive Lebesgue measure if $s+t>n$.
This implies that a suitable affine copy of $B_1$
has all the properties we wanted for $B$. 
\end{proof}

\begin{proof}[Proof of (1)$\Rightarrow$(2) in Lemma~\ref{topsecret}]
Let $A\su\R^n$ be a non-empty Borel set of Hausdorf 
dimension greater than $s$.
By a theorem of Davies \cite{Da52}, $A$ contains a 
compact subset of Hausdorf 
dimension greater than $s$, so we can suppose
that $A$ is compact.
By Lemma~\ref{standard} there exists a compact set 
$B_1\su\R^n$ of upper Minkowski dimension $n-s$ such 
that $A+B$ has positive Lebesgue measure.
By applying (1) of Lemma~\ref{topsecret} to the compact set
$A+B$ we obtain a compact set $E_1$ of upper Minkowski
dimension zero such that $(A+B)+E_1$ has non-empty
interior. Then $E=B+E_1$ is a compact set of
upper Minkowski dimension $n-s$ such that $A+E$
has non-empty interior. 
\end{proof}


It remains to prove (1) of Lemma~\ref{topsecret}.
Let $d_j$ and $m_j$ be increasing sequences of integers. Let $1\le m_j \le d_j^n$. We now define a random compact set. 
For each positive integer $d$ we divide $[0,1]^n$ into $d^n$ grid cubes of equal side length
and denote by $f^d_1,\ldots,f^d_{d^n}$ the homothetic
maps that map $[0,1]^n$ onto these cubes.
Consider the finite sequences 
$$
(i_1,\ldots,i_k)\in
\{1,\ldots,m_1\}\times\ldots\times\{1,\ldots,m_k\}
$$
and to each of them we assign a random number
$j(i_1,\ldots,j_k)\in\{1,\ldots,d_k^n\}$
by choosing for each 
$(i_1,\ldots,i_{k-1})$ 
distinct $j(i_1,\ldots,i_{k-1},1),\ldots,
j(i_1,\ldots,i_{k-1},m_k)$ randomly such that
each selection has probability ${d_k^n \choose m_k}^{-1}$. 
Let
$$
E_k=\bigcup_{i_1=1}^{m_1}\ldots\bigcup_{i_k=1}^{m_k}
f^{d_1}_{j(i_1)}\circ f^{d_2}_{j(i_1,i_2)}\ldots\circ f^{d_k}_{j(i_1,\ldots,i_k)}
([0,1]^n)
$$ 
and let $E=\bigcap E_k$. 


With this construction, the upper Minkowski dimension of $E$ (always) satisfies
\beq\label{dimE}
\limsup_k \frac{\log(m_1\ldots m_k)}{\log(d_1 \ldots d_k)} \le \overline\dim_M(E) \le \limsup_k \frac{\log((m_1\ldots m_k) m_{k+1})}{\log((d_1 \ldots d_k) m_{k+1}^{1/n})}.
\eeq
We will only use the upper bound. 

Fix $$d_k=2^k$$
and
$$m_k=\min(k^{2n+5}, d_k^n);$$
this is a choice that will work but it is not optimal in any sense.
Then the random compact set $E$ has Minkowski dimension zero by \eqref{dimE}.
Thus the following theorem will complete the proof
of Lemma~\ref{topsecret}.


\begin{theorem}\label{t:key}
Let $A\subset \R^n$ be a 
Lebesgue measurable
set of positive Lebesgue measure. Let $E$ be the random compact set as described above. Then $A+E$ has non-empty interior almost surely.
\end{theorem}

\begin{proof}
It is enough to show that the probability in question is positive since Kolmogorov's zero--one law can be applied: We claim that it is a tail event that the set $A+E$ has non-empty interior. That is, if we change the construction of $E$ arbitrarily in finitely many levels to obtain $E'$, then $A+E$ has non-empty interior if and only if $A+E'$ has non-empty interior. Indeed, $E'$ can be covered by finitely many translates of $E$ and vice versa. Therefore the same applies to $A+E$ and $A+E'$. The Baire category theorem concludes the argument.

The following statement enables us to replace the set $A$ with another that has uniform lower bounds on the local densities around every point.

\begin{lemma}\label{maximal}
Let $r_k$, $\eps_k$ be decreasing
sequences of positive reals tending to zero such that 
$$\sum \eps_k <\infty.$$
Let $A_0\subset \R^n$ be a 
Lebesgue measurable set of positive measure. 
Then there is a non-empty 
compact subset $A\subset A_0$ 
such that
\beq\label{r0}
\exists k_0 \ge 1 \ \ 
\forall k\ge k_0 \  \ \forall x\in A \ \  \lambda(A\cap 
(x+[-r_k,r_k]^n)
)\ge (2r_k)^n \eps_k.
\eeq
\end{lemma}

\begin{proof}
By regularity we may assume that $A_0$ is compact. 
Note that if \eqref{r0} holds for $A$ then it also
holds for its closure, so we do not need to 
guarantee the compactness of $A$.

By modifying $\eps_k$, we may assume that each $r_k$ is an integer power of $1/2$ and it is enough to prove that the density of $A$ in every closed dyadic grid cube of side length $r_k$ it intersects is at least $\eps_k$ (for every sufficiently large $k$). 

By applying Lebesgue's density theorem,
choose $k_0$ and a dyadic grid cube $I$ of side length $r_{k_0}$ such that
\begin{equation}\label{e:felesnegyed}
\lambda(A_0\cap I)\ge \lambda(I)/2 \qquad \textrm{and} \qquad
\sum_{k=k_0}^\infty \eps_k < 1/4,
\end{equation}
and let $A_1=A_0\cap I$.
We define a finite or infinite sequence of nested sets $A_1 \supset A_2 \supset\cdots$ recursively.
Suppose that $A_j$ is already defined.
Take a maximal (closed) dyadic grid cube 
$J_j\subset I$ with side length $r_{k_j}$ such that $J_j$ intersects $A_j$ and 
the density of $A_j$ in $J_j$ is less than $\eps_{k_j}$.
If there is no such cube, then we let 
$A=A_{j}$. Otherwise let $A_{j+1}=A_j\sm J_j$ and continue the procedure.
If this procedure does not terminate after finitely many
steps then let $A=\cap_{j=1}^{\infty} A_j$. 

If the sequence terminates after finitely many steps, then the density of $A$ in every dyadic grid cube of side length $r_k$ it intersects is at least $\eps_k$ for 
every $k\ge k_0$. Otherwise, although the sequence $k_j$ is not necessarily
non-decreasing, it must be unbounded since there are only
finitely many dyadic grid cubes of given size in $I$.
Using also that we always took maximal cubes, this implies
that the density of $A$ in every dyadic grid cube of side length $r_k$ it intersects is at least $\eps_k$ for 
every $k\ge k_0$.

It remains to prove that $A$ is non-empty.
Using first the fact that
for each $k$ the subtracted cubes $J_j$ of side length
$r_k$ are nonoverlapping subcubes of 
$I$ and then \eqref{e:felesnegyed}, we obtain
$$\lambda(A) \ge \lambda(A_1) - \lambda(I)\sum_{k=k_0}^\infty \eps_k  \ge \lambda(I)/2 - \lambda(I)/4 >0.$$
%
\end{proof}


\bigskip
We continue with the proof of Theorem~\ref{t:key}.
Recall that $d_k$ and $m_k$ are already defined.
Let 
$$
\eps_k=1/k^2
\qquad \textrm{and} \qquad
r_k=\frac{\eps_k}{20nd_1\cdots d_k}.
$$
By applying Lemma~\ref{maximal}, 
we may assume that the set $A$ is compact
for which \eqref{r0} holds. 

Our aim is, essentially, to show that if $$A+E_k \supset [0,1]^n,$$ then with large probability, $$A+E_{k+1} \supset [0,1]^n.$$
We will prove a similar implication, but we will need to consider some additional sets. Let
$$\alpha_k=1-\eps_k/(10n),$$
and 
$$\beta_k=1-2\eps_k/(10n).$$
For each $k\ge 0$, consider the $m_1\cdots m_k$ many cubes forming $E_k$ and replace each of these with another cube of the same midpoint but of side length $\alpha_k/(d_1\cdots d_k)$. Take their union to obtain the set $E'_k$. We also form $E''_k$ from $E_k$ in the same way but replacing $\alpha_k$ with $\beta_k$.


\begin{claim}
There exists $k_0$ such that for every $k\ge k_0$,
$$\mathbb{P}\big(A+E'_{k+1} \supset A+E'_k)\ge 1-2^{-k}.$$
In fact the same estimate holds conditionally if we fix $E_k$ arbitrarily.
\end{claim}

\begin{proof}

Assume that $E_k$ is fixed. Fix any $x\in A+E'_k$. First we will estimate the probability that $$x\in A+E''_{k+1}.$$

Since $x\in A+E'_k$, we can fix $a\in A$ and $e\in E'_k$ such that $x=a+e$. Let $I'$ be the cube of $E'_k$ that contains $e$. Let $I$ be the corresponding cube of $E_k$.
%
Consider the $d_{k+1}^n$ many grid cubes that $I$ is divided into during the construction of the random set $E_{k+1}$. Let these be $I_j$ ($1\le j\le d_{k+1}^n$). We will denote the corresponding $\beta_k$-shrinked versions of these by $I''_j$. Then the random set $$E''_{k+1}\cap I$$ contains the union of $m_{k+1}$ many of these cubes $I''_j$ chosen uniformly randomly.

Denoting side lengths of cubes by $|\cdot|$, now consider the cube
$$J=[-|I|\eps_k/(20n), \ |I|\eps_k/(20n)]^n=[-r_k,r_k]^n.$$
By the choice of $\alpha_k$, we have $$I'+J=I.$$
Since $e\in I'$, we have
$$e+J\subset I.$$
Let $\I$ be the collection of those cubes $I_j$ that are contained in $e+J$.
Similarly, let $\I''$ be the collection of those cubes $I''_j$ that are contained by $e+J$.
Then we have
$$\lambda\left((e+J) \cap \bigcup \I\right) \ge 
(|J|-2|I|/d_{k+1})^n = 
|J|^n \left(1-\frac{20n}{\eps_k d_{k+1}}\right)^n$$
and
\begin{eqnarray}
\label{a1}
\lambda\left((e+J) \cap \bigcup \I''\right)
&\ge & \beta_k^n |J|^n\left(1-\frac{20n}{\eps_k d_{k+1}}\right)^n
\ge |J|^n\left(\beta_k-\frac{20n}{\eps_k d_{k+1}}\right)^n 
\nonumber \\
&\ge & |J|^n\left(1-\frac{3\eps_k}{10n}\right)^n
\ge |J|^n\left(1-\frac{3\eps_k}{10}\right),
\end{eqnarray}
assuming, in the penultimate inequality, that $k$ (and thus $d_k$) is large enough so that
$$\frac{20n}{\eps_k d_{k+1}} < \frac{\eps_k}{10n}$$
holds.

If $k$ is large enough, by \eqref{r0}, we have
$$\lambda(A\cap (a-J)) \ge \eps_k |J|^n.$$
Using that $x=a+e$ and that Lebesgue measure is isometry invariant, we obtain
\beq\label{a2}
\lambda\Big((e+J)\cap (x-A)\Big) \ge \eps_k |J|^n.
\eeq
Combining \eqref{a1} and \eqref{a2} we obtain
\beq\label{a3}
\lambda\Big((e+J)\cap (x-A) \cap \bigcup \I''\Big) \ge (\eps_k-3\eps_k/10)|J|^n \ge 7\eps_k |J|^n/10 
= c_n \eps_k^{n+1} |I|^n,
\eeq
where $c_n=7/(10^{n+1}n^n)$.
This implies that $x-A$ intersects at least 
$$
c_n\eps_k^{n+1} d_k^n
$$
many cubes $I''_j$ (of $\I''$). If $E''_{k+1}$ contains one of these, say $I''_j$, then
$$\emptyset \neq (x-A) \cap I''_j \subset (x-A) \cap E''_{k+1}$$  
and therefore $x\in A+E''_{k+1}$.
Since by construction $E''_{k+1}$ contains 
uniformly randomly chosen $m_{k+1}$ many of 
$I_1'',\ldots,I''_{d_{k+1}^n}$,
as in \eqref{e:binom} we obtain that
the probability that $E''_{k+1}$ does not contain any of 
the at least $c_n \eps_k^{n+1} d_{k+1}^n$ many 
$I_j''$ that intersect $x-A$ is at most
$$(1-c_n \eps_k^{n+1})^{m_{k+1}}.$$ 
Therefore the probability of $x\notin A+E''_{k+1}$ is at most 
\beq\label{Px}
(1-c_n \eps_k^{n+1})^{m_{k+1}}.
\eeq


Now let
$$\delta=\eps_{k+1}/(20n d_1 \cdots d_{k+1}).$$
Comparing $\delta$ with $\alpha_{k+1}$ and $\beta_{k+1}$ we obtain 
\beq\label{b1}
B(E''_{k+1}, \delta) \subset E'_{k+1}.
\eeq
 
For simplicity, assume that $A\subset [0,1]^n$. Let $S$ be a maximal $\delta$-separated subset of $A+E'_k\subset [0,2]^n$.
Then $$|S|\le C_n\delta^{-n}$$ for some constant $C_n$ depending on $n$ only, and we have 
\beq\label{b2}
A+E'_k \subset B(S, \delta).
\eeq

Now note that if $S\subset A+E''_{k+1}$ (which may or may not hold), then (by \eqref{b1} and \eqref{b2}),
$$A+E'_k \subset B(S,\delta) \subset B(A+E''_{k+1}, \delta)  = A+B(E''_{k+1},\delta) \subset A + E'_{k+1}.$$

Now vary $x$ over the points of $S$. By \eqref{Px}, the probability that $S\subset A+E''_{k+1}$ is at least
$$1-|S|(1-c_n \eps_k^{n+1})^{m_{k+1}}.$$
Therefore 
$$\mathbb{P}(A+E'_k \subset A+E'_{k+1}) \ge 1-|S|(1-c_n \eps_k^{n+1})^{m_{k+1}}$$
where the probability can be understood conditional on fixing $E_k$ arbitrarily.
The choice of $d_k=2^k$ and $m_k=k^{2n+5}$ (for large enough $k$) gives that
$$|S|\le C'_n \,2^{nk^2} / \eps_{k+1}^n = 
C'_n \,2^{nk^2} (k+1)^{2n}$$
and
$$(1-c_n \eps_k^{n+1})^{m_{k+1}} 
= \left(1-\frac{7}{(10k^2)^{n+1}n^n}\right)^{(k+1)^{2n+5}} 
\le 2^{-C''_n k^3},$$
where the constants $C'_n$, $C''_n$ (and later $C'''_n$
as well) are positive and depend only on $n$.
Therefore 
$$\mathbb{P}(A+E'_k \subset A+E'_{k+1}) 
\ge 1-C'_n(k+1)^{2n} 2^{nk^2} 2^{-C''_n k^3}  \ge 1- 2^{-C'''_n k^3} \ge 1-2^{-k}$$
when $k$ is large enough. This finishes the proof of the Claim.
\end{proof}

Applying the claim for $k\ge k_0$, we obtain that with positive probability,
$$A+E'_{k_0} \subset A+E'_{k_0+1} \subset \cdots.$$
Therefore, with positive probability,
$$A+E'_{k_0} \subset A+E.$$
Since $E'_{k_0}$ and thus $A+E'_{k_0}$ have non-empty interior, this finishes the proof of Theorem~\ref{t:key}.
\end{proof}


\bigskip

\begin{remark}
It is possible to reduce this problem on the group $\R$ to a similar one on the group
 $$G = \prod_k \Z / d_k\Z.$$
Then the above argument can be simplified as there is no need to introduce the sets $E'_k$ or $E''_k$.
\end{remark}
\semmi{
\subsection{Option 1}
We may assume that the same density bound applies in each dyadic interval that intersects $A'$ and has diameter at most $r_0$. We may assume $A=A'\subset [0,1]^n$, $r_0=1$.

Let $$T=\sum_k \{-1,0,1\}\frac{1}{d_1\ldots d_k}$$
Then $\dim_B T =0$ and $$A+E+T\supset A+_\text{no carry} E.$$
Using this, we may reduce the problem on $\R$ to a similar problem on the group $$G = \prod_k \Z / d_k\Z.$$
After this point, there are less technical details to handle.

\subsection{Option 2}
We just work with $\R$.
}

\end{document}